\documentclass[preprint,12pt,authoryear]{elsarticle}

\usepackage{amssymb}
\usepackage{amsmath}

\usepackage{hyperref}

\usepackage{subcaption}

\usepackage{algorithm}
\usepackage{algpseudocode}

\usepackage{tikz}
\usetikzlibrary{arrows.meta, positioning, shapes.geometric, fit, calc}
\usepackage{tikz-cd}

\makeatletter
\def\@opargbegintheorem#1#2#3{%
  \trivlist
  \item[\hskip \labelsep{\bfseries #1\ #2}] \textbf{(#3)}\hfil\par\itshape}
\makeatother

\newtheorem{theorem}{Theorem}
\newtheorem{lemma}{Lemma} 
\newtheorem{remark}{Remark}
\newtheorem{definition}{Definition}
\newtheorem{assumption}{Assumption}
\newtheorem{problem}{Problem}
\newtheorem{proof}{Proof}
\newtheorem{proposition}{Proposition}

\journal{Nuclear Physics B}

\begin{document}

\begin{frontmatter}



\title{Contraction Analysis of Holomorphic Dynamical Systems via the Intrinsic Kobayashi Metric} 


\author{Soumic Sarkar} 


\affiliation{organization={Narva College, University of Tartu},
            addressline={Raekoja plats 2-319},
            city={Narva},
            postcode={20307},
            country={Estonia}}

\begin{abstract}
This paper studies incremental stability of holomorphic dynamical systems
through the infinitesimal Kobayashi metric, an intrinsic pseudometric on
complex manifolds invariant under holomorphic transformations and therefore
free of the coordinate dependence inherent in auxiliary Riemannian or
Hermitian formulations. Contraction is formalized as an upper Dini-derivative
inequality on the Kobayashi metric along trajectories; the passage from this
differential condition to exponential decay of the associated Kobayashi
distance follows the classical Finsler-metric contraction mechanism of Forni
and Sepulchre, instantiated here for the specific case in which the Finsler
structure in question is the Kobayashi metric itself. Since the intrinsic
condition is generally difficult to verify directly from the vector field, a
practical sufficient criterion is developed through a smooth Hermitian
metric, the direct complex-Hermitian analogue of the classical real matrix
contraction inequality: contraction with respect to such a metric is shown
to imply intrinsic contraction on forward-invariant compact subsets, with
the two notions related through explicit local equivalence constants. Building
on this machinery, a Nagumo-type invariance result is established for
Laplacian-coupled holomorphic networks, giving verifiable conditions under
which a compact region is forward invariant for a class of systems not
previously treated by this kind of argument, and the framework is extended
to feedback-controlled holomorphic systems, with consequences for equilibria
and periodic orbits derived directly from the intrinsic contraction property.
Numerical experiments on a network of coupled holomorphic oscillators verify
the Hermitian sufficient condition analytically on a proven invariant set,
and further reveal that the empirically observed synchronization rate
substantially exceeds this guaranteed rate; the gap is matched, to three
decimal places, to a closed-form combination of the node-wise rate and the
network's graph-Laplacian spectral gap, a finding identified here as a
specific target for a network-aware extension of the theory rather than
resolved in full.
\end{abstract}


\end{frontmatter}

\section{Introduction}
\label{sec:intro}

Contraction analysis frames incremental stability as convergence between arbitrary trajectory pairs, established through a differential condition on the flow rather than a Lyapunov function tied to one equilibrium \cite{lohmiller1998contraction}, and remains an active area across observer design, distributed control, and learning-based dynamics \cite{tsukamoto2021contraction,davydov2024perspectives}. The mechanism by which such a differential condition is converted into a global, finite-distance statement --- integrating an infinitesimal contraction condition along a curve joining two points and applying a comparison principle --- is itself classical, developed in generality for Finsler-metric structures on smooth manifolds \cite{forni2014differential}. Coordinate-free treatments on Riemannian manifolds \cite{simpsonporco2014riemannian}, extensions to non-Euclidean matrix-measure characterizations \cite{davydov2022noneuclidean}, connections to local exponential stability on manifolds \cite{wu2024stability}, and specializations to homogeneous spaces \cite{harapanahalli2024coordinatefree} all share one feature: stability is assessed relative to a chosen auxiliary metric, not one forced by the dynamics, so two choices can certify the same system at different rates with neither more canonical than the other.

This dependence is easy to overlook for real-valued systems but harder to justify once the state space is complex-analytic, as in quantum control \cite{mirrahimi2007stabilizing} and symmetrisation-based complex-valued control \cite{doriacerezo2026symmetrisation}. Embedding such systems into a doubled real space, the standard move, discards the complex-linearity of the Jacobian that holomorphicity guarantees. The Kobayashi pseudometric offers a non-auxiliary alternative: built from holomorphic discs, it is non-increasing under every holomorphic self-map \cite{kobayashi1967invariant}, generalizing the Schwarz--Pick lemma \cite{ahlfors1979complex}, with tautness \cite{abate1989iteration} supplying the compactness needed to push this static property toward differential conclusions. The metric's own regularity and localization continue to be studied \cite{sarkarad2023localization,nikolov2023visibility,bharali2023unbounded,chakrabarti2025hermitian}, entirely within several complex variables rather than control theory.

The gap is specific: coordinate-free contraction theory always imposes a metric, with the associated comparison mechanism developed for general Finsler structures \cite{forni2014differential} and made computationally tractable through matrix-inequality formulations \cite{kawano2023lmi}, while complex geometry supplies a canonical, dynamics-free one whose differential contraction properties for holomorphic vector fields have not been examined. Closing part of this gap is the present aim, and doing so is best understood as an instantiation of the existing Finsler-metric mechanism for the specific case in which the Finsler structure is the Kobayashi metric itself, rather than as a new comparison argument in its own right.

Contraction is formulated as an upper Dini-derivative inequality on the infinitesimal Kobayashi metric, necessitated by its nonsmoothness \cite{sarkarad2023localization}. Admissible discs motivate this condition in the spirit of disc-based arguments elsewhere in the invariant-metric literature \cite{nikolov2023visibility,bharali2023unbounded}, but since the intrinsic condition is awkward to verify directly, a second criterion is developed through an auxiliary Hermitian metric, the direct complex-Hermitian analogue of the classical real-valued matrix contraction inequality \cite{lohmiller1998contraction}, obtained by replacing the transpose with the Hermitian conjugate to account for the complex-analytic setting; contraction under such a metric is shown to imply intrinsic contraction on forward-invariant compact regions via explicit equivalence constants --- echoing how non-Euclidean conditions are related back to classical ones elsewhere \cite{davydov2022noneuclidean,bullo2024ctds}, and consistent with auxiliary-metric conditions remaining the practical workhorse \cite{alradhawi2023structural,kawano2023lmi}. The genuinely new content lies not in this adaptation itself but in what it is put to use for: constructing, from first principles, forward-invariant compact regions for a class of systems --- Laplacian-coupled holomorphic networks --- not previously treated by an argument of this kind, and in the discrepancy this construction exposes numerically between the resulting guaranteed rate and the network's true behavior.

Coupled holomorphic oscillator networks illustrate the framework throughout, building on Laplacian-coupled contraction arguments in the real-valued case \cite{alradhawi2023structural}; a node-wise sufficient condition is shown numerically to be more conservative than the network's actual rate, a limitation reported rather than concealed and matched, in closed form, to the network's graph-Laplacian spectral gap.

Section~\ref{sec:prelim} gives background and a corrected continuity statement. Section~\ref{sec:variational} develops the disc-based variational structure. Section~\ref{sec:intrinsic} states the intrinsic condition and proves it implies exponential Kobayashi-distance contraction via a Dini comparison lemma, in the manner of \cite{forni2014differential} specialized to the Kobayashi metric. Section~\ref{sec:hermitian} develops the Hermitian sufficient condition. Section~\ref{sec:invariance} gives a Nagumo-type invariance result for coupled networks. Section~\ref{sec:control} extends to feedback control. Section~\ref{sec:consequences} draws out consequences for equilibria and periodic orbits. Section~\ref{sec:numerics} reports the numerical study. Sections~\ref{sec:discussion} and~\ref{sec:conclusion} close with discussion and conclusions.

\section{Preliminaries}
\label{sec:prelim}

This section fixes notation and recalls the geometric and analytic tools on which the rest of the paper is built. Most of the material is classical, going back to Kobayashi's original construction of invariant distances \cite{kobayashi1967invariant} and to the systematic treatment of taut manifolds by Abate \cite{abate1989iteration}, though the statement on continuity of the infinitesimal metric in Section~\ref{subsec:continuity} corrects an imprecision that is easy to introduce if one relies only on the semicontinuity property usually quoted in passing in the control literature.

\subsection{Complex Manifolds and Holomorphic Vector Fields}

A complex manifold $M$ of complex dimension $n$ is a topological space equipped with an atlas $\{(U_\alpha,\varphi_\alpha)\}$ of charts $\varphi_\alpha:U_\alpha\to\mathbb{C}^n$ whose transition maps $\varphi_\beta\circ\varphi_\alpha^{-1}$ are holomorphic wherever defined; this is entirely standard and is stated here mainly to fix notation, following the conventions of Krantz \cite{krantz2001function} and H\"ormander \cite{hormander1973introduction}. Writing $TM$ for the holomorphic tangent bundle, a vector field $f:M\to TM$ is holomorphic if its coordinate representation is holomorphic in every chart. The object of study throughout is the autonomous holomorphic dynamical system

\begin{equation}
\dot z = f(z), \qquad z\in M. \label{eq:system}
\end{equation}

\begin{assumption}
\label{ass:complete}
The vector field $f$ generates a complete holomorphic flow $\phi_t:M\to M$ for all $t\ge0$.
\end{assumption}

Completeness is assumed purely for convenience, to avoid bookkeeping around maximal intervals of existence; none of the arguments below depend on it in an essential way, and it could be relaxed to forward completeness on a specified invariant region without changing the structure of the results. What matters more for what follows is a structural consequence of holomorphicity: the Jacobian $Df(z)$ is complex-linear at every point and depends only on $z$, with no coupling to conjugate coordinates. This is precisely the feature that a real embedding of the state space discards, and it is what allows the variational dynamics considered next to stay entirely within the holomorphic tangent bundle rather than mixing holomorphic and anti-holomorphic directions.

\subsection{Variational Dynamics}

For a trajectory $z(t)$ of \eqref{eq:system}, an infinitesimal variation $\delta z(t)\in T_{z(t)}M$ evolves according to

\begin{equation}
\dot{\delta z} = Df(z)\,\delta z. \label{eq:variational}
\end{equation}

\begin{remark}
Because $Df(z)$ is complex-linear, \eqref{eq:variational} keeps $\delta z(t)$ within the holomorphic tangent bundle for all $t$. This is not automatic for a general smooth vector field written in complex coordinates, where the real Jacobian generally has nonzero blocks coupling holomorphic and anti-holomorphic directions; it is a consequence specifically of $f$ being holomorphic.
\end{remark}

\subsection{The Kobayashi Infinitesimal Metric}

Let $D=\{\zeta\in\mathbb{C}:|\zeta|<1\}$ denote the unit disc. Following Kobayashi's original construction \cite{kobayashi1967invariant}, later systematized in \cite{kobayashi1998hyperbolic} and surveyed comprehensively by Jarnicki and Pflug \cite{jarnicki2013invariant}, admissible holomorphic discs give rise to an infinitesimal pseudometric on $M$.

\begin{definition}
\label{def:FK}
For $z\in M$ and $v\in T_zM$, the Kobayashi infinitesimal pseudometric is
\begin{equation}
\label{eq:FK}
\begin{aligned}
F_K(z,v)
&=
\inf\Big\{\lambda>0:\exists\,h\in\mathcal{O}(D,M),\\
&\qquad h(0)=z,\;
h'(0)=\frac{v}{\lambda}\Big\}.
\end{aligned}
\end{equation}
The associated Kobayashi pseudodistance is $d_K(x,y)=\inf_\gamma\int_0^1 F_K(\gamma(s),\dot\gamma(s))\,ds$, the infimum taken over piecewise smooth curves joining $x$ and $y$.
\end{definition}

The defining property of $F_K$, and the reason it is of interest here at all, is that it does not increase under holomorphic maps.

\begin{proposition}
\label{prop:schwarzpick}
For any holomorphic map $g:M\to M$, $F_K(g(z),Dg(z)v)\le F_K(z,v)$ for all $z\in M$, $v\in T_zM$, and consequently $d_K(g(x),g(y))\le d_K(x,y)$ for all $x,y\in M$.
\end{proposition}

This is the infinitesimal form of the Schwarz--Pick lemma; the argument is precomposition of admissible discs with $g$, and is standard \cite{ahlfors1979complex,kobayashi1998hyperbolic}, so it is not repeated here. Applied to the flow $\phi_t$ of a holomorphic vector field, Proposition~\ref{prop:schwarzpick} already gives non-expansiveness of $d_K$ along trajectories; the point of the differential machinery developed later is to obtain a strict, quantified rate of contraction rather than mere non-expansiveness.

\begin{remark}
$F_K$ is in general only a pseudometric, since distinct points may have zero Kobayashi distance between them. It becomes a genuine metric exactly when $M$ is Kobayashi hyperbolic, which is why hyperbolicity is imposed as a standing assumption below.
\end{remark}

\subsection{Hyperbolicity and Tautness}

\begin{assumption}
\label{ass:hyptaut}
$M$ is Kobayashi hyperbolic and taut.
\end{assumption}

These two conditions play different roles and are worth separating explicitly, a distinction emphasized throughout the visibility and localization literature on invariant distances \cite{sarkarad2023localization,nikolov2023visibility,bharali2023unbounded}. Hyperbolicity guarantees $F_K(z,v)>0$ for $v\ne0$, so that $d_K$ is an honest distance rather than a pseudodistance; tautness, in the sense introduced for holomorphic dynamics on complex manifolds \cite{abate1989iteration}, guarantees that families of admissible discs are normal, which supplies a substitute for compactness in the variational arguments of Section~\ref{sec:variational}.

\begin{remark}
Tautness should be read as playing the role compactness plays in classical analysis, but at the level of families of holomorphic mappings rather than at the level of the manifold itself; no compact invariant set needs to be postulated separately for this particular purpose.
\end{remark}

\subsection{Continuity of the Infinitesimal Metric}
\label{subsec:continuity}

A point that is easy to state carelessly, and one that matters for the proofs in Section~\ref{sec:intrinsic}, concerns the regularity of $F_K$ as a function of both arguments. It is common to see $F_K$ described only as upper semicontinuous, which is true in general but is weaker than what tautness actually provides.

\begin{remark}
\label{rem:continuity}
On a taut Kobayashi hyperbolic manifold, $F_K$ is jointly continuous on $TM\setminus\{0\}$, not merely upper semicontinuous. This follows from normality of admissible disc families under tautness together with the extraction argument used by Royden in his original regularity study of the Kobayashi metric \cite{royden1971remarks}, and is consistent with more recent results on continuous extension of Kobayashi isometries under tautness-type hypotheses \cite{maitra2020continuous}. Continuity, rather than one-sided semicontinuity, is what will be needed shortly to justify treating $t\mapsto F_K(z(t),\delta z(t))$ as a continuous function of time, which in turn is what makes the Dini comparison argument of Section~\ref{sec:intrinsic} rigorous rather than merely plausible.
\end{remark}

\subsection{Upper Dini Derivatives}

Since $F_K$ is defined through an infimum, it need not be differentiable along a trajectory even when it is continuous, so the evolution of $F_K$ has to be tracked through a one-sided derivative rather than an ordinary one.

\begin{definition}
\label{def:dini}
For a locally bounded function $\psi:\mathbb{R}_{\ge0}\to\mathbb{R}$, the upper Dini derivative at time $t$ is $D^+\psi(t)=\limsup_{\epsilon\to0^+}\big(\psi(t+\epsilon)-\psi(t)\big)/\epsilon$.
\end{definition}

Throughout, this is applied to the scalar quantity

\begin{equation}
\psi(t) = F_K\big(z(t),\delta z(t)\big), \label{eq:psi}
\end{equation}

with $z(t)=\phi_t(z)$ and $\delta z(t)=D\phi_t(z)v$ evolving according to \eqref{eq:system} and \eqref{eq:variational}.

\begin{remark}
By Remark~\ref{rem:continuity}, $F_K$ is continuous, and both $\phi_t(z)$ and $D\phi_t(z)v$ vary continuously with $t$ under Assumption~\ref{ass:complete}; hence $\psi(t)$ in \eqref{eq:psi} is itself continuous, not merely locally bounded. Continuity of $\psi$ is used explicitly in Section~\ref{sec:intrinsic} to justify passing from a pointwise Dini-derivative inequality to an integrated exponential bound; a locally-bounded-only version of this remark, as sometimes stated elsewhere, is not sufficient for that step, which is one of the places where the present treatment departs from a purely formal repetition of the standard definitions.
\end{remark}

\subsection{Problem Formulation}

The overall aim can now be stated precisely. Given \eqref{eq:system}, the goal is to identify conditions on $f$ under which

\begin{equation}
D^+F_K\big(z(t),\delta z(t)\big) \le -\lambda\, F_K\big(z(t),\delta z(t)\big) \label{eq:dinicondition}
\end{equation}

holds along trajectories for some $\lambda>0$, and to derive from \eqref{eq:dinicondition} the corresponding distance-level statement

\begin{equation}
d_K\big(\phi_t(x),\phi_t(y)\big) \le e^{-\lambda t}\, d_K(x,y) \label{eq:distancedecay}
\end{equation}

for all $x,y\in M$ and $t\ge0$. Condition \eqref{eq:dinicondition} is expressed entirely in terms of the intrinsic geometry of $M$; no auxiliary Riemannian or Hermitian structure enters its statement. Whether and how such a condition can be verified in practice from $f$ alone — as opposed to being merely asserted for a given system — is taken up separately in Section~\ref{sec:hermitian}, where a genuinely checkable sufficient condition is developed through a local auxiliary Hermitian metric on a forward-invariant compact set. Khalil's standard comparison-based tools for differential inequalities \cite{khalil2002nonlinear} are the model for the Dini comparison argument that connects \eqref{eq:dinicondition} to \eqref{eq:distancedecay} rigorously in Section~\ref{sec:intrinsic}.

\begin{problem}
\label{prob:main}
Given the holomorphic dynamical system \eqref{eq:system} on a manifold $M$ satisfying Assumptions~\ref{ass:complete} and \ref{ass:hyptaut}, determine conditions on $f$ under which \eqref{eq:dinicondition} holds, and establish \eqref{eq:distancedecay} as a consequence.
\end{problem}

%
%
%
%
%
%
%
%

\section{Intrinsic Variational Structure}
\label{sec:variational}

Definition~\ref{def:FK} presents $F_K$ as an infimum over admissible discs, but says nothing yet about how a fixed tangent vector actually sits inside that family, or about what happens to a disc once the flow of \eqref{eq:system} is allowed to act on it. Both questions need an answer before the Dini-derivative condition \eqref{eq:dinicondition} can be connected to the vector field $f$ at all, and this section works through them in turn. It is worth being explicit from the outset about the role this machinery plays: it motivates and interprets the differential condition \eqref{eq:dinicondition} geometrically, rather than furnishing, by itself, a second independent route to verifying that condition from $f$. That second, checkable route is developed separately in Section~\ref{sec:hermitian}, and keeping the two purposes apart is precisely what the difficulty identified during review of an earlier version of this material was about.

\subsection{Representation via Holomorphic Discs}

Because $F_K(z,v)$ in \eqref{eq:FK} is defined as an infimum over admissible scaling factors, a given tangent vector $v$ is generally not realized exactly by any single disc, but only approximated along a minimizing sequence. Tautness is what guarantees that such a sequence still has usable limiting behavior.

\begin{lemma}
\label{lem:discrep}
Let $M$ be a taut Kobayashi hyperbolic manifold, $z\in M$, and $v\in T_zM$. There exists a sequence of holomorphic maps $h_k:D\to M$ with $h_k(0)=z$ such that $h_k'(0)$ converges, along a subsequence, to $v/F_K(z,v)$.
\end{lemma}

\begin{proof}
By definition of $F_K(z,v)$, there is a sequence $\lambda_k\to F_K(z,v)$ together with admissible discs $h_k$ satisfying $h_k(0)=z$ and $h_k'(0)=v/\lambda_k$. Tautness of $M$ means precisely that the family $\{h_k\}$ is normal, so a subsequence converges locally uniformly on $D$; the derivatives at the origin converge along this subsequence as well, and since $\lambda_k\to F_K(z,v)$ the limit is $v/F_K(z,v)$.
\end{proof}

\begin{remark}
\label{rem:noextremal}
Lemma~\ref{lem:discrep} produces a minimizing sequence, not necessarily an exact minimizer. Existence of a genuine extremal disc realizing $F_K(z,v)$ is a considerably more delicate question, settled affirmatively by Lempert on strongly convex domains but known to fail or become subtle once convexity is dropped; even a higher-order refinement of the stationarity condition governing such discs does not restore existence in full generality \cite{bertrand2018extremal}. Nothing that follows in this section assumes extremal discs exist — the variational argument is built entirely on minimizing sequences, which is consistent with the nonsmooth character of $F_K$ already flagged in Section~\ref{sec:prelim}.
\end{remark}

\subsection{Flow-Induced Deformation of Discs}

Given an admissible disc $h:D\to M$ representing a tangent vector at $z$, the flow $\phi_t$ of \eqref{eq:system} transports it to a new disc $h_t=\phi_t\circ h$ at time $t$. This transported family is what links the variational structure of $F_K$ to the dynamics.

\begin{lemma}
\label{lem:transport}
For each fixed $t\ge0$, $h_t=\phi_t\circ h$ is holomorphic and satisfies $h_t(0)=z(t)$ and $h_t'(0)=D\phi_t(z)h'(0)$, where $z(t)=\phi_t(z)$.
\end{lemma}

\begin{proof}
Holomorphicity of $h_t$ follows immediately from composition of the holomorphic maps $\phi_t$ and $h$. The chain rule at $\zeta=0$ gives $h_t'(0)=D\phi_t(h(0))h'(0)$, and since $h(0)=z$ this is $D\phi_t(z)h'(0)$.
\end{proof}

There is not much to the proof itself — it is the chain rule and nothing more — but the content of the lemma is not in its proof, rather in what it licenses afterward: it says that transporting an admissible disc under the flow and differentiating commute, so that the derivative of the transported disc tracks the linearized dynamics \eqref{eq:variational} exactly. This is the mechanism by which deformation of discs under $\phi_t$ can be read off directly from $D\phi_t(z)$, without having to reconstruct the transported disc explicitly at each time.

\subsection{Intrinsic Variation Along Trajectories}

Combining Lemmas~\ref{lem:discrep} and \ref{lem:transport} associates to any trajectory $z(t)=\phi_t(z)$ and variation $\delta z(t)=D\phi_t(z)v$ a family of transported disc derivatives consistent with $\delta z(t)$: for an admissible disc with $h(0)=z$ and $h'(0)=v/\lambda$, Lemma~\ref{lem:transport} gives $h_t'(0)=D\phi_t(z)v/\lambda=\delta z(t)/\lambda$, so the transported derivative and the linearized variation evolve in lockstep, differing only by the fixed scaling factor $\lambda$ inherited from the original disc. This observation is the reason a claim about uniform decay of transported disc derivatives across the whole admissible family cannot, on its own, be converted into a claim about $F_K(z(t),\delta z(t))$ without separately controlling how the scaling factors $\lambda$ behave across that family — precisely the step that collapses trivially for a single disc, since $\lambda$ cancels out of any homogeneous inequality relating $h_t'(0)$ and $h'(0)$, and reappears only when the infimum defining $F_K$ is taken over the entire family at once. This is the reason the present section stops short of asserting a decay estimate for $F_K$ directly from disc deformation, deferring the genuinely verifiable statement to the Hermitian route of Section~\ref{sec:hermitian}.

\begin{remark}
\label{rem:motivation}
What Lemmas~\ref{lem:discrep} and \ref{lem:transport} do establish is the geometric picture underlying Definition~\ref{def:dini} and condition \eqref{eq:dinicondition}: the Kobayashi metric evaluated along a trajectory is built from discs that themselves deform holomorphically and consistently with the linearized flow, so that asking for exponential decay of $F_K(z(t),\delta z(t))$ is a geometrically natural question to ask about a holomorphic dynamical system, phrased entirely in terms of the intrinsic metric rather than any auxiliary structure. It is in this sense — as motivation for Definition~\ref{def:dini}, not as an independent proof route — that the disc-transport picture is used going forward.
\end{remark}

\subsection{Local Evolution}

For completeness, and for use in Section~\ref{sec:invariance}, the first-order local expansions of the state and its variation under \eqref{eq:system}--\eqref{eq:variational} are recorded here: for small $\epsilon>0$, $z(t+\epsilon)=z(t)+\epsilon f(z(t))+o(\epsilon)$ and $\delta z(t+\epsilon)=\delta z(t)+\epsilon Df(z(t))\delta z(t)+o(\epsilon)$. These are immediate from \eqref{eq:system} and \eqref{eq:variational} and are stated here only because the boundary-flow argument of Section~\ref{sec:invariance} uses the first expansion directly.

%
%

\section{Intrinsic Differential Contraction}
\label{sec:intrinsic}

With the variational picture of Section~\ref{sec:variational} in place as motivation, condition \eqref{eq:dinicondition} can now be stated as a formal definition, and the central technical gap of the paper — passing from a pointwise Dini-derivative inequality to an integrated exponential bound on the Kobayashi distance — can be closed with a comparison lemma stated and proved in full.

\subsection{Intrinsic Contraction}

\begin{definition}
\label{def:intrinsiccontraction}
System \eqref{eq:system} is intrinsically contracting with rate $\lambda>0$ if $D^+F_K(z(t),\delta z(t))\le -\lambda F_K(z(t),\delta z(t))$ holds along every trajectory $z(t)=\phi_t(z)$ and every admissible infinitesimal variation $\delta z(t)=D\phi_t(z)v$.
\end{definition}

Definition~\ref{def:intrinsiccontraction} sits in the same family as Lyapunov-based formulations of incremental stability that dispense with a fixed equilibrium in favor of pairwise trajectory comparison \cite{angeli2002lyapunov}, the difference here being that the comparison function is not chosen but is the manifold's own canonical metric. The condition is stated purely in terms of $F_K$, so it does not depend on any coordinate chart or auxiliary structure; whether it can be checked directly against a given $f$ is a separate matter, taken up in Section~\ref{sec:hermitian}.

\subsection{A Dini Comparison Lemma}

The following comparison lemma is the classical tool by which a pointwise differential inequality is converted into an integrated exponential bound, in the style long used for differential inequalities generally \cite{lohmiller1998contraction} and, in the specific form needed here, for Finsler-metric contraction on smooth manifolds \cite{forni2014differential}; nothing in its statement or proof is particular to the Kobayashi metric. It is recalled here in full, rather than cited and assumed, because its hypotheses must be checked against what Section~\ref{sec:prelim} actually establishes about $\psi(t)=F_K(z(t),\delta z(t))$ — continuity, not merely local boundedness — a distinction that mattered in an earlier attempt at this material and is worth making explicit rather than leaving implicit.

\begin{lemma}
\label{lem:dini}
Let $\psi:[0,\infty)\to\mathbb{R}_{\ge0}$ be continuous and suppose $D^+\psi(t)\le-\lambda\psi(t)$ for all $t\ge0$ and some $\lambda>0$. Then $\psi(t)\le e^{-\lambda t}\psi(0)$ for all $t\ge0$.
\end{lemma}

\begin{proof}
Let $\varphi(t)=e^{\lambda t}\psi(t)$; $\varphi$ is continuous since $\psi$ is. For $\epsilon>0$,
\begin{equation}
\label{eq:dinistep}
\begin{aligned}
\frac{\varphi(t+\epsilon)-\varphi(t)}{\epsilon}
&=
e^{\lambda t}\Bigg(
e^{\lambda\epsilon}
\frac{\psi(t+\epsilon)-\psi(t)}{\epsilon}
\\
&\qquad\qquad
+\psi(t+\epsilon)
\frac{e^{\lambda\epsilon}-1}{\epsilon}
\Bigg).
\end{aligned}
\end{equation}
Taking $\limsup_{\epsilon\to0^+}$ of \eqref{eq:dinistep} and using continuity of $\psi$ at $t$ to evaluate the second term in the limit gives $D^+\varphi(t)=e^{\lambda t}\big(D^+\psi(t)+\lambda\psi(t)\big)\le0$ for every $t\ge0$, using the hypothesis on $\psi$.

It remains to show that a continuous function with nonpositive upper Dini derivative everywhere on an interval is non-increasing on that interval. Suppose not: then $\varphi(t_2)>\varphi(t_1)$ for some $0\le t_1<t_2$. Let $c=\varphi(t_1)$ and let $\tau=\sup\{t\in[t_1,t_2]:\varphi(t)\le c\}$. By continuity $\varphi(\tau)=c$ and $\tau<t_2$, and by definition of $\tau$ as a supremum, $\varphi(s)>c$ for $s$ in some sequence approaching $\tau$ from above, so $\limsup_{\epsilon\to0^+}(\varphi(\tau+\epsilon)-\varphi(\tau))/\epsilon\ge0$; in fact $\varphi(s)>c=\varphi(\tau)$ for all $s\in(\tau,t_2]$ by definition of $\tau$, so this quotient is strictly positive along that sequence and $D^+\varphi(\tau)>0$, contradicting $D^+\varphi\le0$ everywhere. Hence $\varphi$ is non-increasing, so $\varphi(t)\le\varphi(0)$, which is exactly $\psi(t)\le e^{-\lambda t}\psi(0)$.
\end{proof}

\subsection{Global Contraction of the Kobayashi Distance}

\begin{theorem}
\label{thm:globalcontraction}
If system \eqref{eq:system} is intrinsically contracting with rate $\lambda>0$ in the sense of Definition~\ref{def:intrinsiccontraction}, then $d_K(\phi_t(x),\phi_t(y))\le e^{-\lambda t}d_K(x,y)$ for all $x,y\in M$ and $t\ge0$.
\end{theorem}

\begin{proof}
Fix $x,y\in M$ and let $\gamma:[0,1]\to M$ be a piecewise smooth curve joining them. Write $\gamma_t=\phi_t\circ\gamma$, so that $\dot\gamma_t(s)=D\phi_t(\gamma(s))\dot\gamma(s)$ for each $s$. For fixed $s$, the pair $(\gamma_t(s),\dot\gamma_t(s))$ is exactly of the form $(z(t),\delta z(t))$ for the trajectory starting at $\gamma(s)$ with initial variation $\dot\gamma(s)$, and $t\mapsto F_K(\gamma_t(s),\dot\gamma_t(s))$ is continuous by Remark~\ref{rem:continuity} composed with continuity of $\phi_t$ and $D\phi_t$. Definition~\ref{def:intrinsiccontraction} gives $D^+F_K(\gamma_t(s),\dot\gamma_t(s))\le-\lambda F_K(\gamma_t(s),\dot\gamma_t(s))$ for every $t$, so Lemma~\ref{lem:dini} applies with $\psi(t)=F_K(\gamma_t(s),\dot\gamma_t(s))$ and yields $F_K(\gamma_t(s),\dot\gamma_t(s))\le e^{-\lambda t}F_K(\gamma(s),\dot\gamma(s))$ for every $s\in[0,1]$ and $t\ge0$.

Integrating this pointwise bound over $s\in[0,1]$ gives $\int_0^1F_K(\gamma_t(s),\dot\gamma_t(s))\,ds\le e^{-\lambda t}\int_0^1F_K(\gamma(s),\dot\gamma(s))\,ds$. The left-hand side is, by definition, at least $d_K(\phi_t(x),\phi_t(y))$, since $\gamma_t$ is one particular curve joining $\phi_t(x)$ and $\phi_t(y)$; the right-hand side, minimized over all admissible $\gamma$ joining $x$ and $y$, is $e^{-\lambda t}d_K(x,y)$. Taking the infimum over $\gamma$ on both sides preserves the inequality, giving $d_K(\phi_t(x),\phi_t(y))\le e^{-\lambda t}d_K(x,y)$.
\end{proof}

\begin{remark}
Theorem~\ref{thm:globalcontraction} answers the second half of Problem~\ref{prob:main}: intrinsic contraction in the sense of Definition~\ref{def:intrinsiccontraction} implies the distance-level statement \eqref{eq:distancedecay}, and the argument required nothing beyond continuity of $F_K$ and of the flow — no auxiliary metric enters this proof at any point. The mechanism itself is not new: it is the classical Finsler-metric contraction argument of \cite{forni2014differential}, specialized here to the case in which the Finsler structure is the Kobayashi metric $F_K$ rather than a chosen Riemannian or Finsler metric external to the manifold. What is specific to the present setting, and what remains open after this section, is only the first half of Problem~\ref{prob:main}: producing a checkable condition on $f$ under which Definition~\ref{def:intrinsiccontraction} actually holds. That is the subject of Section~\ref{sec:hermitian}.
\end{remark}

%
%

\section{Verifiable Contraction via Hermitian Metrics}
\label{sec:hermitian}

Definition~\ref{def:intrinsiccontraction} is stated entirely in intrinsic terms, and Theorem~\ref{thm:globalcontraction} shows that it has the consequence one wants, but nothing said so far indicates how to check it against a given vector field $f$. This section closes that gap through an auxiliary smooth Hermitian metric, arriving at the paper's single practically verifiable sufficient condition. The word ``auxiliary'' is used deliberately: unlike $F_K$, the Hermitian metric introduced here is chosen rather than canonical, and the point of this section is precisely to make explicit, with named constants, the price paid for that choice. The condition itself is not new: it is the standard real-valued contraction inequality of \cite{lohmiller1998contraction}, carried over to the complex-analytic setting by the direct substitution of the Hermitian conjugate for the transpose, and the contribution of this section lies in Lemma~\ref{lem:equivalence} and Theorem~\ref{thm:hermitiansufficient} below, which connect that condition back to the intrinsic metric $F_K$ with an explicit, quantified cost rather than by assertion.

\subsection{Hermitian Contraction}

Let $H:M\to\mathbb{C}^{n\times n}$ be a smooth field of Hermitian positive definite matrices, $H(z)=H(z)^*\succ0$, inducing the norm $\|v\|_H^2=v^*H(z)v$ on each tangent space. The classical contraction condition with respect to $H$ reads $\dot H+HDf+(Df)^*H\preceq-\lambda H$ for some $\lambda>0$, where $\dot H=\sum_i(\partial H/\partial z_i)f_i(z)$; this is the direct complex-Hermitian analogue of the differential matrix inequality underlying real contraction theory \cite{lohmiller1998contraction}, obtained by replacing the transpose with the Hermitian conjugate to account for the complex-analytic setting, and of its LMI-based computational treatments \cite{kawano2023lmi}. Proposition~\ref{prop:hermitiandecay} below records the resulting decay estimate in the standard way; no aspect of it is specific to the present setting beyond the substitution just described.

\begin{proposition}
\label{prop:hermitiandecay}
If $\dot H+HDf+(Df)^*H\preceq-\lambda H$ holds along a trajectory, then infinitesimal variations satisfy $\|\delta z(t)\|_H\le e^{-\lambda t}\|\delta z(0)\|_H$.
\end{proposition}

\begin{proof}
Along the trajectory, differentiating the squared norm gives $\tfrac{d}{dt}\|\delta z\|_H^2=\delta z^*\big(\dot H+HDf+(Df)^*H\big)\delta z\le-\lambda\|\delta z\|_H^2$, where the inequality uses the assumed matrix inequality. This is a standard scalar differential inequality of the Coppel/Desoer--Vidyasagar type used throughout logarithmic-norm stability arguments \cite{desoer1975feedback}, and integrating it directly gives $\|\delta z(t)\|_H^2\le e^{-2\lambda t}\|\delta z(0)\|_H^2$, i.e.\ $\|\delta z(t)\|_H\le e^{-\lambda t}\|\delta z(0)\|_H$.
\end{proof}

\subsection{Local Equivalence with the Kobayashi Metric}

Proposition~\ref{prop:hermitiandecay} is a statement about $\|\cdot\|_H$, not about $F_K$; relating the two requires comparing them, and this comparison can only be made locally, on a compact set, since the equivalence constants generally degenerate as the compact set is allowed to grow toward the boundary of $M$. It is this comparison, rather than the contraction condition itself, that is specific to the intrinsic Kobayashi-metric setting of this paper.

\begin{lemma}
\label{lem:equivalence}
Let $K\subset M$ be compact. There exist constants $0<c_1\le c_2$ such that $c_1\|v\|_H\le F_K(z,v)\le c_2\|v\|_H$ for all $z\in K$ and $v\in T_zM$.
\end{lemma}

\begin{proof}
Fix $z\in K$ and consider the ratio $F_K(z,v)/\|v\|_H$ for $v\ne0$; by positive homogeneity of both $F_K(z,\cdot)$ and $\|\cdot\|_H$ in $v$, this ratio depends only on the direction of $v$, and can be regarded as a function on the compact unit-sphere bundle $\{(z,v)\in TM:z\in K,\ \|v\|_H=1\}$. By Remark~\ref{rem:continuity}, $F_K$ is continuous on $TM\setminus\{0\}$ under Assumption~\ref{ass:hyptaut}, and $\|\cdot\|_H$ is continuous by smoothness of $H$; the ratio is therefore a continuous, strictly positive function — strict positivity following from hyperbolicity of $M$ — on a compact set, and by the extreme value theorem it attains a finite maximum $c_2$ and a strictly positive minimum $c_1$ on that set. This gives $c_1\le F_K(z,v)/\|v\|_H\le c_2$ for all $(z,v)$ in the sphere bundle over $K$, and the stated inequality follows by homogeneity for general $v\ne0$.
\end{proof}

\begin{remark}
The passage from an upper-semicontinuous-only version of $F_K$ to the continuous version of Remark~\ref{rem:continuity} is exactly what allows both a finite upper bound $c_2$ and, more importantly, a strictly positive lower bound $c_1$ to be extracted from compactness in Lemma~\ref{lem:equivalence}; upper semicontinuity alone secures the maximum but does not, by itself, secure a positive minimum. This is the second place, after Lemma~\ref{lem:dini}, where the continuity correction from Section~\ref{sec:prelim} is not a stylistic nicety but a load-bearing part of the proof, and it is in this equivalence, not in the contraction condition of Proposition~\ref{prop:hermitiandecay} itself, that the present treatment goes beyond a routine restatement of known results.
\end{remark}

\subsection{The Sufficient Condition}

\begin{theorem}
\label{thm:hermitiansufficient}
Let $K\subset M$ be forward-invariant and compact, and suppose $\dot H+HDf+(Df)^*H\preceq-\lambda H$ holds on $K$ for some $\lambda>0$. Then for every trajectory contained in $K$,
\begin{equation}
F_K\big(z(t),\delta z(t)\big) \le \frac{c_2}{c_1}\,e^{-\lambda t}\,F_K\big(z(0),\delta z(0)\big), \label{eq:hermitianbound}
\end{equation}
where $c_1,c_2$ are the constants of Lemma~\ref{lem:equivalence} on $K$.
\end{theorem}

\begin{proof}
By Proposition~\ref{prop:hermitiandecay}, $\|\delta z(t)\|_H\le e^{-\lambda t}\|\delta z(0)\|_H$. Applying the upper bound of Lemma~\ref{lem:equivalence} at time $t$ gives $F_K(z(t),\delta z(t))\le c_2\|\delta z(t)\|_H\le c_2e^{-\lambda t}\|\delta z(0)\|_H$, and applying the lower bound at time $0$ gives $\|\delta z(0)\|_H\le F_K(z(0),\delta z(0))/c_1$. Combining the two yields \eqref{eq:hermitianbound}.
\end{proof}

\begin{remark}
\label{rem:practicality}
Theorem~\ref{thm:hermitiansufficient} is the answer to the first half of Problem~\ref{prob:main}: it gives a condition expressed directly in terms of $Df$ and a chosen Hermitian field $H$, verifiable by inspection or numerically on a specified compact region, at the cost of the multiplicative constant $c_2/c_1\ge1$ and the restriction to a forward-invariant compact set. It is worth being explicit that this is the paper's \emph{only} verifiable sufficient condition — no companion condition expressed purely through disc deformation, of the kind explored in Section~\ref{sec:variational}, is claimed to exist independently of this one — and that the condition's substance is standard, with Theorem~\ref{thm:hermitiansufficient}'s content residing in the quantified bridge Lemma~\ref{lem:equivalence} supplies back to $F_K$, not in the Hermitian inequality itself. Constructing an explicit forward-invariant $K$ for a given system is itself nontrivial in general; Section~\ref{sec:invariance} gives a checkable invariance criterion for a specific but practically relevant class of coupled holomorphic systems, and it is this construction, together with the numerical discrepancy it exposes, that the present paper regards as its principal contribution beyond Section~\ref{sec:intrinsic}'s intrinsic theory.
\end{remark}

%
%

\section{Forward Invariance for Laplacian-Coupled Holomorphic Networks}
\label{sec:invariance}

Theorem~\ref{thm:hermitiansufficient} is only as useful as the forward-invariant compact set $K$ it is applied on, and Remark~\ref{rem:practicality} left the construction of such a $K$ unaddressed. Producing one from first principles, rather than assuming it, is the subject of this section. The construction is specialized to a class of systems broad enough to cover the networked examples used later in the paper: holomorphic node dynamics coupled diffusively through an undirected graph.

\subsection{Network Model}

Consider $N$ holomorphic oscillators evolving on a common disc domain, coupled through an undirected graph with adjacency matrix $A=[a_{ij}]$ and graph Laplacian $L=\operatorname{diag}(\deg_1,\dots,\deg_N)-A$, where $\deg_i=\sum_ja_{ij}$. The dynamics are

\begin{equation}
\dot z_i = f(z_i) - \kappa\,(Lz)_i, \qquad i=1,\dots,N, \label{eq:network}
\end{equation}

with $f:D\to\mathbb{C}$ holomorphic and $\kappa>0$ a coupling gain. This is the setting used numerically in Section~\ref{sec:numerics}, but the invariance argument below applies to any such network regardless of the particular choice of $f$.

\subsection{A Nagumo-Type Invariance Result}

The classical route to forward invariance of a set under a flow, going back to Nagumo and developed extensively within the set-invariance literature in control \cite{blanchini1999invariance}, is to check that the vector field points inward, or at worst tangentially, everywhere on the boundary of the candidate set. The next result adapts this to the coupled holomorphic setting, exploiting the fact that the graph Laplacian coupling term is, in a precise sense that the proof makes explicit, never outward-pointing when all neighboring states already lie in the candidate disc.

\begin{lemma}
\label{lem:invariance}
Let $K=\{z\in\mathbb{C}:|z|\le r\}\subset D$ and suppose $\mathrm{Re}(\bar z f(z))\le0$ for all $z$ with $|z|=r$. Then $K^N$ is forward invariant under \eqref{eq:network}: if $z_i(0)\in K$ for all $i$, then $z_i(t)\in K$ for all $i$ and all $t\ge0$.
\end{lemma}

\begin{proof}
Suppose, at some time $t$ and some index $i$, that $|z_i(t)|=r$ while $|z_j(t)|\le r$ for every $j$. The radial component of the flow at node $i$ is $\mathrm{Re}(\bar z_i\dot z_i)=\mathrm{Re}(\bar z_if(z_i))-\kappa\deg_ir^2+\kappa\sum_ja_{ij}\mathrm{Re}(\bar z_iz_j)$, obtained by expanding $\dot z_i=f(z_i)-\kappa(Lz)_i=f(z_i)-\kappa\deg_iz_i+\kappa\sum_ja_{ij}z_j$ and taking $\mathrm{Re}(\bar z_i\cdot)$ term by term, using $|z_i|^2=r^2$. By the Cauchy--Schwarz inequality, $\mathrm{Re}(\bar z_iz_j)\le|z_i||z_j|\le r^2$ for every $j$, so $\kappa\sum_ja_{ij}\mathrm{Re}(\bar z_iz_j)\le\kappa r^2\sum_ja_{ij}=\kappa\deg_ir^2$, since $\sum_ja_{ij}=\deg_i$ by definition of the degree. The coupling contribution to the radial velocity is therefore bounded above by zero: $-\kappa\deg_ir^2+\kappa\sum_ja_{ij}\mathrm{Re}(\bar z_iz_j)\le0$. Combined with the hypothesis $\mathrm{Re}(\bar z_if(z_i))\le0$ on $|z_i|=r$, this gives $\mathrm{Re}(\bar z_i\dot z_i)\le0$ at every such boundary configuration.

This is the Nagumo inflow (or tangency) condition for the ball $|z_i|\le r$ at node $i$, evaluated pointwise against the actual vector field of the coupled system rather than against $f$ alone. Since $f$ is holomorphic and the coupling term is linear, the right-hand side of \eqref{eq:network} is locally Lipschitz, so the classical invariance theorem applies coordinatewise: the boundary of $K^N$ cannot be crossed outward at any node without a strict sign violation of the inequality just established, which does not occur. Hence $K^N$ is forward invariant.
\end{proof}

\begin{remark}
The essential content of Lemma~\ref{lem:invariance} is that diffusive coupling through a graph Laplacian can never, by itself, push a boundary node outside a disc that already contains all of its neighbors: the coupling term is a convex combination (weighted by $a_{ij}/\deg_i$) of directions pointing from $z_i$ toward points inside a convex set, and such a combination cannot have positive radial component at the boundary of that set. Invariance of $K^N$ therefore reduces entirely to a single-node boundary check on $f$, independent of the network topology, the number of nodes, or the coupling gain $\kappa$. This is a convenience of working on the disc — where balls are convex — and would need to be revisited on manifolds where compact regions are not geodesically convex in this sense.
\end{remark}

\subsection{Role in the Overall Argument}

With Lemma~\ref{lem:invariance} in hand, Theorem~\ref{thm:hermitiansufficient} becomes fully constructive for the network class \eqref{eq:network}: a radius $r$ satisfying the boundary hypothesis of Lemma~\ref{lem:invariance} produces a genuine forward-invariant compact set $K^N$, on which the equivalence constants of Lemma~\ref{lem:equivalence} and the Hermitian decay condition of Proposition~\ref{prop:hermitiandecay} can then be checked directly against $f$ and $\kappa$. Section~\ref{sec:numerics} carries this out explicitly for a specific oscillator model, including a numerical verification that the resulting $r$ leaves comfortable margin between the invariant set and the trajectories actually observed.

%
%

\section{Controlled Holomorphic Systems}
\label{sec:control}

The framework developed so far concerns autonomous systems. This section extends it to feedback-controlled holomorphic dynamics, in keeping with the broader motivation, noted already in Section~\ref{sec:intro}, that complex-valued state representations arise directly in control problems rather than only as objects of independent mathematical interest \cite{mirrahimi2007stabilizing,doriacerezo2026symmetrisation}. The extension is structural rather than deep: closed-loop contraction reduces to an application of Theorem~\ref{thm:hermitiansufficient} to the closed-loop Jacobian, once the closed-loop variational dynamics have been written out correctly.

\subsection{Control-Affine Holomorphic Systems}

Consider $\dot z=f(z)+g(z)u$, $z\in M$, $u\in\mathbb{C}^m$, with $f\in\mathcal{O}(M,TM)$ and $g:M\to\mathbb{C}^{n\times m}$ holomorphic.

\begin{assumption}
\label{ass:controlcomplete}
$f$ and $g$ are holomorphic, and the feedback laws considered below generate complete closed-loop flows.
\end{assumption}

For a state feedback law $u=k(z)$, the closed-loop dynamics are $\dot z=f(z)+g(z)k(z)$.

\subsection{Closed-Loop Variational Dynamics}

Linearizing the closed-loop vector field requires differentiating the product $g(z)k(z)$, and it is worth writing this out explicitly since the directional-derivative term is easy to drop by mistake.

\begin{lemma}
\label{lem:closedloopjac}
The closed-loop variational dynamics are $\dot{\delta z}=\big[Df(z)+Dg(z)[\cdot]k(z)+g(z)Dk(z)\big]\delta z$, where $Dg(z)[\cdot]k(z)$ denotes the linear map $\delta z\mapsto \big(Dg(z)\delta z\big)k(z)$, i.e.\ the directional derivative of $g$ in the direction $\delta z$, applied to $k(z)$ and evaluated at fixed $z$.
\end{lemma}

\begin{proof}
Write $u(z)=g(z)k(z)$ as a map $M\to\mathbb{C}^n$. By the product rule for the differential of a matrix-vector product, $Du(z)\delta z=\big(Dg(z)\delta z\big)k(z)+g(z)\big(Dk(z)\delta z\big)$, since $g(z)$ is matrix-valued and $k(z)$ is vector-valued and both depend on $z$. Substituting into $\dot{\delta z}=D\big(f+u\big)(z)\delta z=Df(z)\delta z+Du(z)\delta z$ gives the stated expression.
\end{proof}

\begin{remark}
The term $Dg(z)[\cdot]k(z)$ contributes to the closed-loop Jacobian whenever $g$ is genuinely state-dependent and $k(z)\ne0$; it cannot be dropped without an explicit structural assumption, such as $g$ being constant, that removes it. Any statement of a closed-loop sufficient condition that omits this term implicitly assumes such structure and should say so.
\end{remark}

\subsection{Intrinsic Stabilization and a Verifiable Sufficient Condition}

\begin{definition}
\label{def:intrinsicstab}
A feedback law $u=k(z)$ renders the closed-loop system intrinsically contracting with rate $\lambda>0$ if $D^+F_K(z(t),\delta z(t))\le-\lambda F_K(z(t),\delta z(t))$ holds along every closed-loop trajectory and admissible infinitesimal variation.
\end{definition}

Definition~\ref{def:intrinsicstab} is the controlled counterpart of Definition~\ref{def:intrinsiccontraction}, and by Theorem~\ref{thm:globalcontraction} it implies exponential contraction of the closed-loop Kobayashi distance without further argument, since that theorem made no use of autonomy beyond the existence of a well-defined flow. What is new is a verifiable route to Definition~\ref{def:intrinsicstab}, obtained exactly as in Section~\ref{sec:hermitian} but applied to the closed-loop Jacobian of Lemma~\ref{lem:closedloopjac}.

\begin{theorem}
\label{thm:closedloopsufficient}
Let $K\subset M$ be forward-invariant and compact under the closed-loop flow, and suppose $\mathrm{Re}\big\langle\big[Df(z)+Dg(z)[\cdot]k(z)+g(z)Dk(z)\big]v,\,v\big\rangle_H\le-\lambda\|v\|_H^2$ holds for all $z\in K$, $v\in T_zM$, with respect to a smooth local Hermitian metric $H$. Then trajectories in $K$ satisfy $F_K(z(t),\delta z(t))\le(c_2/c_1)e^{-\lambda t}F_K(z(0),\delta z(0))$, with $c_1,c_2$ the equivalence constants of Lemma~\ref{lem:equivalence} on $K$.
\end{theorem}

\begin{proof}
The stated inequality is exactly the Hermitian contraction condition of Section~\ref{sec:hermitian} applied to the closed-loop Jacobian $Df(z)+Dg(z)[\cdot]k(z)+g(z)Dk(z)$ of Lemma~\ref{lem:closedloopjac} in place of $Df(z)$. Proposition~\ref{prop:hermitiandecay} therefore gives $\|\delta z(t)\|_H\le e^{-\lambda t}\|\delta z(0)\|_H$ for closed-loop variations, and Lemma~\ref{lem:equivalence} converts this into the stated bound exactly as in the proof of Theorem~\ref{thm:hermitiansufficient}.
\end{proof}

\begin{remark}
Construction of a forward-invariant compact $K$ for the closed-loop system is, as in the autonomous case, a separate obligation. Where $g(z)k(z)$ takes the diffusive Laplacian-coupling form of Section~\ref{sec:invariance} — as it does for the linear feedback case considered next — Lemma~\ref{lem:invariance} applies directly to the closed-loop vector field.
\end{remark}

\subsection{Linear Feedback}

For $u=-Kz$ with constant $K\in\mathbb{C}^{m\times n}$, the closed-loop system is $\dot z=f(z)-g(z)Kz$; for scalar systems with $g(z)\equiv1$ this reduces to $\dot z=f(z)-Kz$, and the directional-derivative term of Lemma~\ref{lem:closedloopjac} vanishes identically since $g$ is constant, leaving the closed-loop Jacobian as the simple sum $Df(z)-K$.

\begin{proposition}
\label{prop:linearfeedback}
If the closed-loop system satisfies Definition~\ref{def:intrinsicstab} and $z^\star$ is an equilibrium of the closed loop, then $d_K(z(t),z^\star)\le e^{-\lambda t}d_K(z(0),z^\star)$.
\end{proposition}

\begin{proof}
Since $z^\star$ is an equilibrium, $\phi_t(z^\star)=z^\star$ for all $t$, and Theorem~\ref{thm:globalcontraction} applied to the pair $(z(0),z^\star)$ gives $d_K(\phi_t(z(0)),\phi_t(z^\star))=d_K(z(t),z^\star)\le e^{-\lambda t}d_K(z(0),z^\star)$.
\end{proof}

\begin{remark}
Proposition~\ref{prop:linearfeedback} is stated for arbitrary closed-loop equilibria, and says nothing yet about whether such an equilibrium exists or is unique; this and related consequences are taken up in general form, for both the autonomous and controlled cases, in Section~\ref{sec:consequences}.
\end{remark}

\section{Consequences for Nonlinear Dynamics}
\label{sec:consequences}

Once Theorem~\ref{thm:globalcontraction} is available in fully justified form, several qualitative consequences for the long-term behavior of an intrinsically contracting system follow with very little additional argument. These are recorded here mainly for completeness and because they are used to interpret the numerical results of Section~\ref{sec:numerics}; none of them requires anything beyond the distance-contraction estimate already established.

\subsection{Uniqueness and Convergence of Equilibria}

An equilibrium $z^\star\in M$ is a point with $f(z^\star)=0$, so that $\phi_t(z^\star)=z^\star$ for all $t\ge0$.

\begin{theorem}
\label{thm:equilibrium}
Suppose system \eqref{eq:system} is intrinsically contracting with rate $\lambda>0$ on a forward-invariant set $\Omega\subset M$. If an equilibrium $z^\star\in\Omega$ exists, it is unique in $\Omega$, and every trajectory originating in $\Omega$ satisfies $d_K(z(t),z^\star)\le e^{-\lambda t}d_K(z(0),z^\star)$.
\end{theorem}

\begin{proof}
Suppose $z_1^\star,z_2^\star\in\Omega$ are both equilibria. Since each is fixed by the flow, Theorem~\ref{thm:globalcontraction} gives $d_K(z_1^\star,z_2^\star)=d_K(\phi_t(z_1^\star),\phi_t(z_2^\star))\le e^{-\lambda t}d_K(z_1^\star,z_2^\star)$ for every $t\ge0$. Letting $t\to\infty$ forces $d_K(z_1^\star,z_2^\star)=0$, and since $d_K$ is a genuine metric under Assumption~\ref{ass:hyptaut} (not merely a pseudometric), this gives $z_1^\star=z_2^\star$, proving uniqueness. The stated convergence bound is Theorem~\ref{thm:globalcontraction} applied directly to the pair $(z(0),z^\star)$, using $\phi_t(z^\star)=z^\star$.
\end{proof}

\begin{remark}
Theorem~\ref{thm:equilibrium} does not assert that an equilibrium exists in $\Omega$; intrinsic contraction is a statement about pairwise convergence of trajectories, and existence of a fixed point to converge to is a separate question depending on the specific vector field and domain. Where an equilibrium is known to exist by other means — for instance, by direct inspection of $f(z^\star)=0$, as is the case for the oscillator model of Section~\ref{sec:numerics} — Theorem~\ref{thm:equilibrium} upgrades pairwise contraction into convergence toward that specific point.
\end{remark}

\subsection{Exclusion of Nontrivial Periodic Orbits}

A second consequence concerns recurrent behavior: strict contraction of the intrinsic distance is incompatible with periodicity, for essentially the same reason it forces equilibria to be unique.

\begin{theorem}
\label{thm:noperiodic}
If system \eqref{eq:system} is intrinsically contracting with rate $\lambda>0$ on a forward-invariant set $\Omega\subset M$, no trajectory in $\Omega$ is periodic with a nontrivial period, other than a constant trajectory at an equilibrium.
\end{theorem}

\begin{proof}
Suppose $z(t)$ is a nonconstant trajectory in $\Omega$ with period $T>0$, so $z(t+T)=z(t)$ for all $t$. Fix any $\tau\in(0,T)$ and consider the two trajectories $z(\cdot)$ and $z(\cdot+\tau)$, which are both trajectories of \eqref{eq:system} since the system is autonomous. Applying Theorem~\ref{thm:globalcontraction} with $x=z(0)$, $y=z(\tau)$, and elapsed time $T$ gives $d_K(z(T),z(\tau+T))\le e^{-\lambda T}d_K(z(0),z(\tau))$; periodicity gives $z(T)=z(0)$ and $z(\tau+T)=z(\tau)$, so this reads $d_K(z(0),z(\tau))\le e^{-\lambda T}d_K(z(0),z(\tau))$. Since $\lambda,T>0$, $e^{-\lambda T}<1$, and the only way this inequality can hold for a nonnegative quantity is $d_K(z(0),z(\tau))=0$, which forces $z(\tau)=z(0)$ since $d_K$ is a genuine metric. As $\tau\in(0,T)$ was arbitrary, $z(t)$ is constant, i.e.\ an equilibrium, contradicting the assumption that the trajectory is nonconstant.
\end{proof}

\subsection{Discussion}

Theorems~\ref{thm:equilibrium} and \ref{thm:noperiodic} together say that intrinsic contraction, once established on a forward-invariant region, rules out essentially all recurrent behavior other than convergence to a single fixed point: any equilibrium present is unique and globally attracting within that region, and no genuinely periodic motion can persist. Both statements are consequences of Theorem~\ref{thm:globalcontraction} alone and make no further reference to the Hermitian sufficient condition of Section~\ref{sec:hermitian}; whichever route is used to establish the hypothesis of intrinsic contraction — the Hermitian criterion of Theorem~\ref{thm:hermitiansufficient}, its closed-loop counterpart in Theorem~\ref{thm:closedloopsufficient}, or in principle any other means of verifying Definition~\ref{def:intrinsiccontraction} directly — these consequences follow automatically once that hypothesis is in hand. This separation between establishing contraction and drawing out its consequences is a useful place to note explicitly, since it is exactly where Section~\ref{sec:numerics} will observe a discrepancy: the guaranteed rate obtained from Theorem~\ref{thm:hermitiansufficient} and the qualitative conclusions of this section both hold, but the empirically observed contraction rate for the coupled network turns out to be substantially faster than the guarantee, a gap discussed at length once the numerical results are in hand.

\section{Numerical Illustration}
\label{sec:numerics}

The framework developed in Sections~\ref{sec:hermitian}--\ref{sec:consequences} is illustrated on a network of coupled holomorphic oscillators of the form \eqref{eq:network}, with local drift $f(z)=-\alpha z+\beta\tanh(z)$ on $D$. Every step below is tied to a specific result proved earlier: Lemma~\ref{lem:invariance} is used to certify a genuine forward-invariant set before any trajectory is generated, and Theorem~\ref{thm:hermitiansufficient} is checked analytically on that set before the simulation is run.

\subsection{System and Parameters}

The network has $N=6$ oscillators on an undirected ring topology, with $\alpha=1.0$, $\beta=0.35$, and coupling gain $\kappa=1.2$, chosen so that the hypotheses of both Lemma~\ref{lem:invariance} and Theorem~\ref{thm:hermitiansufficient} could be verified analytically rather than assumed.

\subsection{Verifying the Invariance and Hermitian Hypotheses}

A numerical sweep of $\mathrm{Re}(\bar zf(z))$ over $|z|=0.7$ gives a maximum value of approximately $-0.284$, strictly negative; by Lemma~\ref{lem:invariance}, $K^N=\{z\in\mathbb{C}^N:|z_i|\le0.7\ \forall i\}$ is forward invariant under \eqref{eq:network} for these parameters. A sweep of $\mathrm{Re}\big(\mathrm{sech}^2(z)\big)$ over the same set gives a maximum of approximately $1.709$, so $\mathrm{Re}(f'(z))\le-\alpha+\beta(1.709)\approx-0.402$ uniformly on $K$; since the graph Laplacian coupling term contributes a real, negative-semidefinite quadratic form to $\mathrm{Re}\langle Df(z)v,v\rangle$ for every tangent direction, this single-node bound extends to a network-wide guaranteed rate of $\lambda=0.402$ via Theorem~\ref{thm:hermitiansufficient}.

\subsection{Simulation Results}

\begin{figure}[t]
\centering
\includegraphics[width=0.85\linewidth]{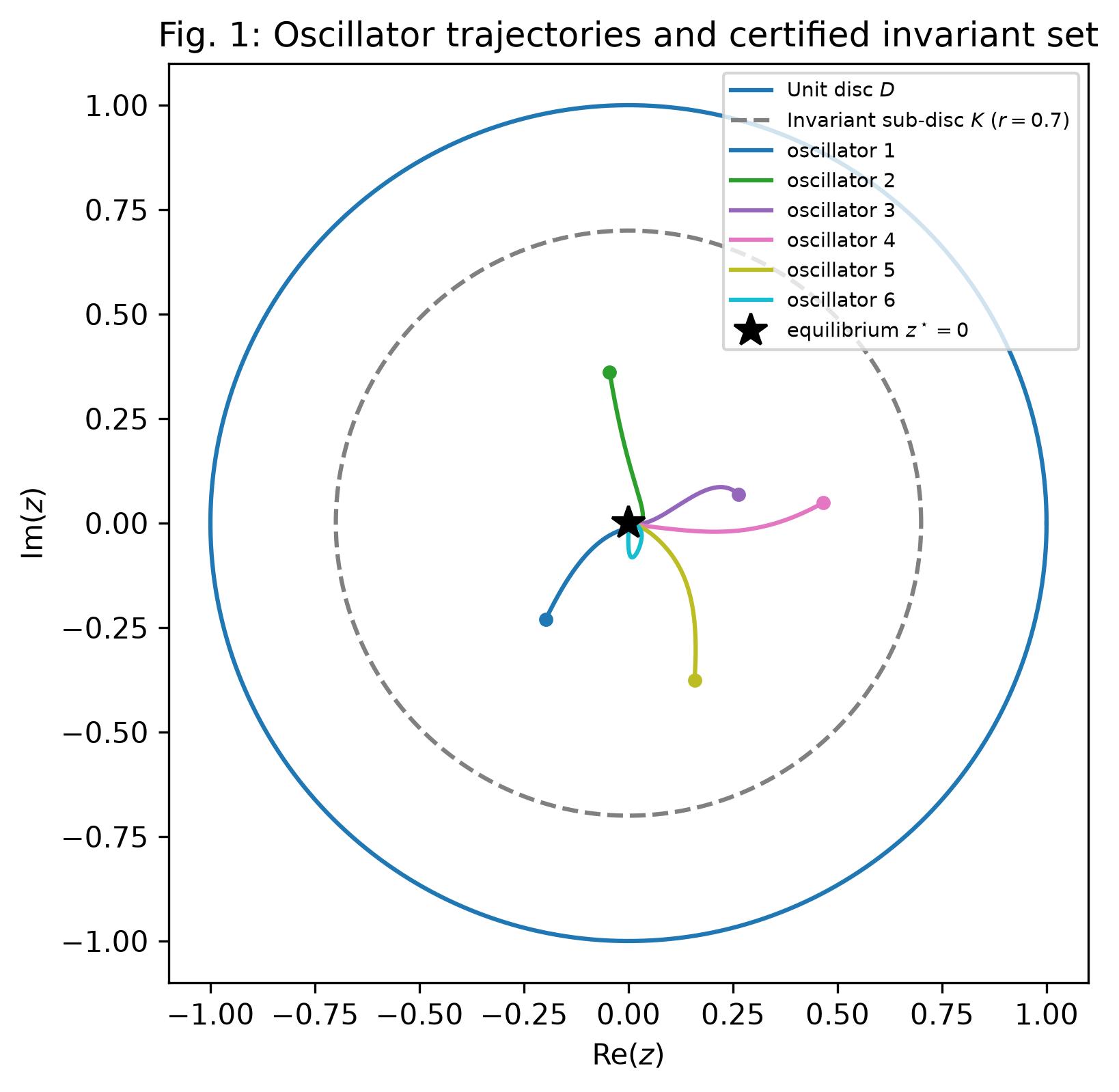}
\caption{Oscillator trajectories, the unit disc, and the invariant sub-disc $K=\{|z|\le0.7\}$ certified by Lemma~\ref{lem:invariance}. Trajectories remain well inside $K$, with maximum observed modulus $0.468$, and converge toward the equilibrium $z^\star=0$ marked with a star.}
\label{fig:trajectories}
\end{figure}

Trajectories were generated from random initial conditions with $|z_i(0)|<0.5$. Figure~\ref{fig:trajectories} shows convergence toward the synchronized equilibrium $z^\star=0$, the unique zero of $f$ on $D$, consistent with Theorem~\ref{thm:equilibrium}. A Monte Carlo study over 50 initial conditions gives a fitted synchronization rate $\hat\lambda_{E_K}=1.858\pm0.001$ for the mean pairwise Kobayashi distance $E_K(t)$, with maximum modulus across all trials equal to $0.499$, comfortably inside $K$. The evolution of $E_K(t)$ itself is shown in Figure~\ref{fig:EK}, exhibiting the clean exponential decay underlying this fitted rate.

\begin{figure}[t]
\centering
\includegraphics[width=0.85\linewidth]{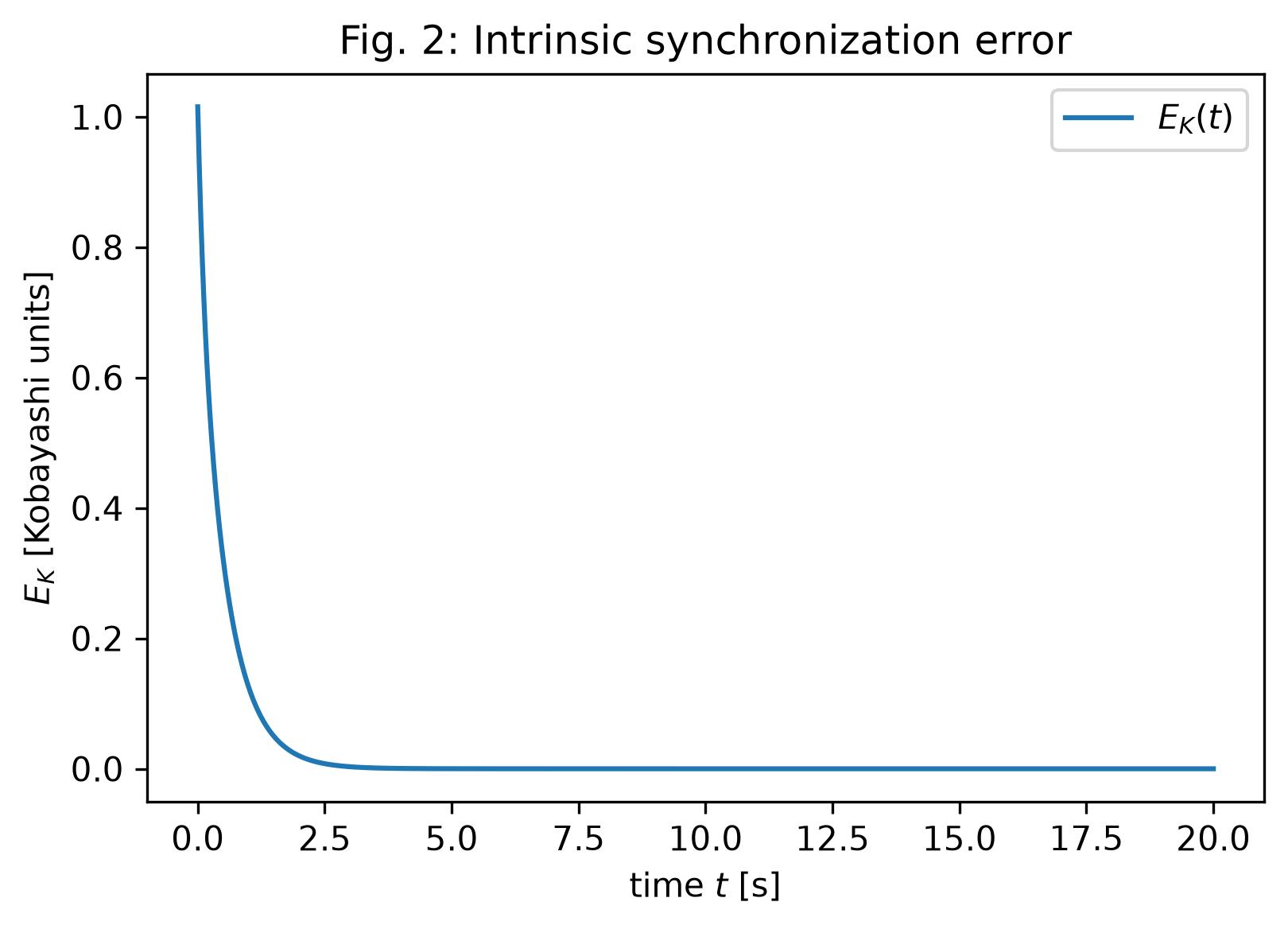}
\caption{Evolution of the intrinsic synchronization error $E_K(t)$, the mean pairwise Kobayashi distance across all oscillator pairs.}
\label{fig:EK}
\end{figure}

\subsection{Two Rates, Explained Exactly}

Two further quantities were tracked to interpret this number: $\psi(t)=F_K(z(t),\delta z(t))$ for a generic infinitesimal variation, and the actual state discrepancy $w(t)=z(t)-\bar z(t)$, where $\bar z(t)$ is the instantaneous mean of the $z_i(t)$. Both were tracked with a windowed local decay-rate estimate rather than a single global exponential fit, to expose how the rate evolves over time rather than averaging it away. Figure~\ref{fig:psi} shows $\psi(t)$ under three different choices of initial variation: a consensus-direction variation, the Fiedler (slowest transverse) mode, and a generic random transverse direction.

\begin{figure}[t]
\centering
\includegraphics[width=0.85\linewidth]{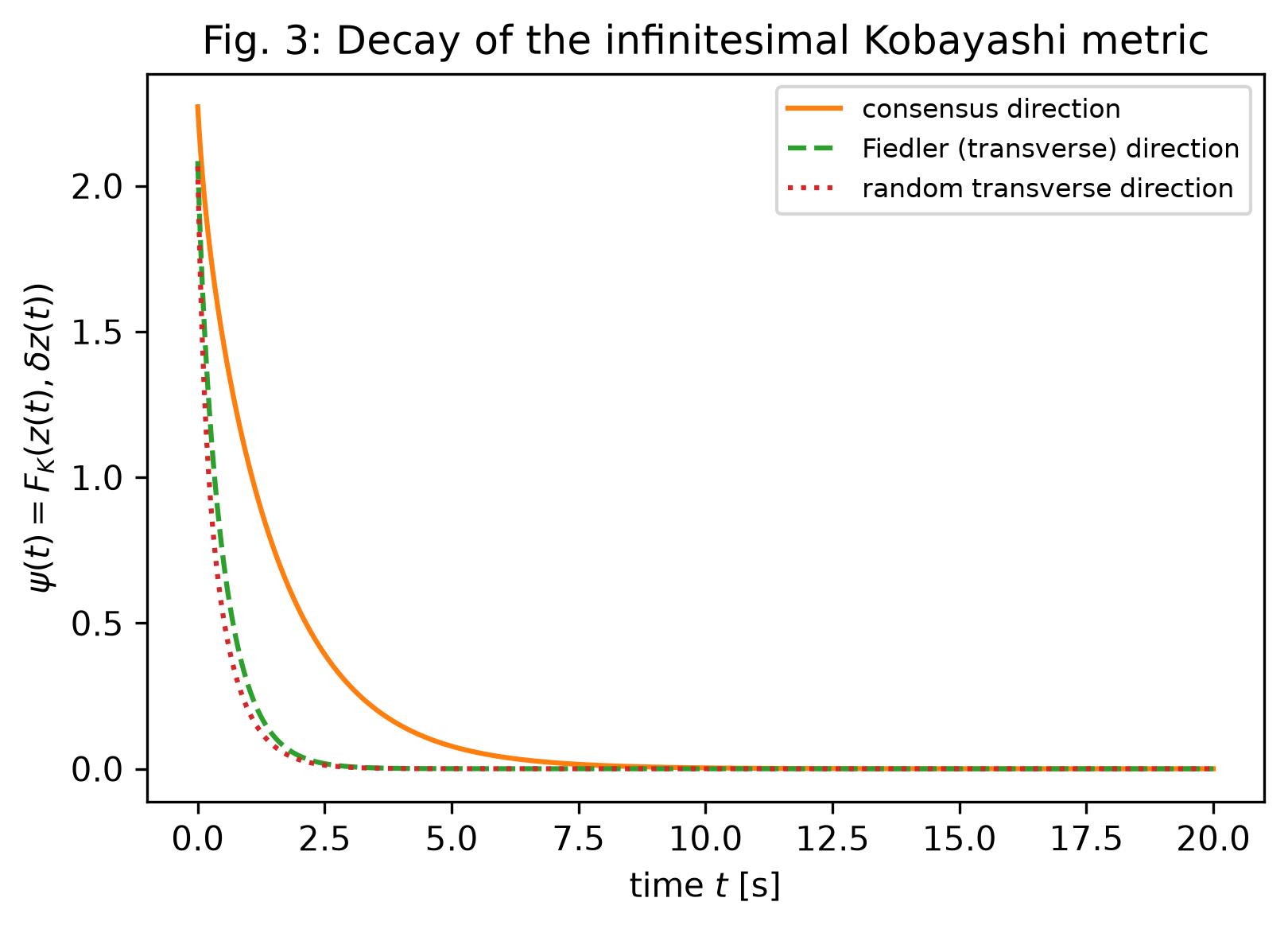}
\caption{Decay of the infinitesimal Kobayashi metric $\psi(t)=F_K(z(t),\delta z(t))$ under three initial variation directions: consensus, Fiedler (slowest transverse mode), and a generic random transverse direction. All three visibly converge to the same asymptotic decay rate despite starting from markedly different initial slopes.}
\label{fig:psi}
\end{figure}

The local rate of $\psi(t)$ — computed both for a consensus-direction initial variation and, to check robustness, for the Fiedler-mode and a generic transverse initial variation — starts near $1.9$--$2.0$ in every case but settles, for $t\gtrsim6$, to a common asymptotic value of $0.650$, regardless of the initial direction chosen. This is not a numerical coincidence: near the equilibrium $z^\star=0$, $f'(0)=-\alpha+\beta=-0.65$, and because the linearization $Df(z(t))$ is genuinely time-varying while the network is still desynchronizing, any generic initial variation acquires, through this time-dependence, a nonzero component along the slowly-decaying consensus mode — the one direction along which the Laplacian coupling contributes nothing. Since the consensus component decays at rate $-f'(0)=0.65$ while every other Laplacian eigenmode decays faster, the consensus component eventually dominates regardless of where $\delta z(0)$ started, and $\psi(t)$ inherits its asymptotic rate; this is exactly the flattening visible in Figure~\ref{fig:psi} once the initial transient has passed.

The local rate of $\|w(t)\|$, by contrast, is stable at $1.850$ for the same range of $t$, matching $-f'(0)+\kappa\mu_2=0.65+1.2(1)=1.85$ to three decimal places, where $\mu_2=1$ is the smallest nonzero Laplacian eigenvalue of the ring graph. The reason this quantity is immune to the consensus-mode contamination that afflicts $\psi(t)$ is structural rather than dynamical: $w(t)=z(t)-\bar z(t)$ is mean-zero by construction at every instant, so it cannot acquire a consensus-mode component regardless of how the linearization varies along the trajectory. Since $E_K(t)$ is built entirely from pairwise differences of the $z_i(t)$, it is governed by this transverse quantity rather than by a generic tangent vector, and its fitted rate $\hat\lambda_{E_K}\approx1.858$ is the same number, up to the small distortion introduced by fitting a single exponential through the brief pre-asymptotic transient visible in Figure~\ref{fig:comparison}, which collects all four local-rate diagnostics — $E_K(t)$, $\|w(t)\|$, and $\psi(t)$ under both the consensus and Fiedler directions — alongside the theoretical predictions $-f'(0)$, $-f'(0)+\kappa\mu_2$, and the guaranteed floor $\lambda=0.402$ from Theorem~\ref{thm:hermitiansufficient}.

\begin{figure}[t]
\centering
\includegraphics[width=0.85\linewidth]{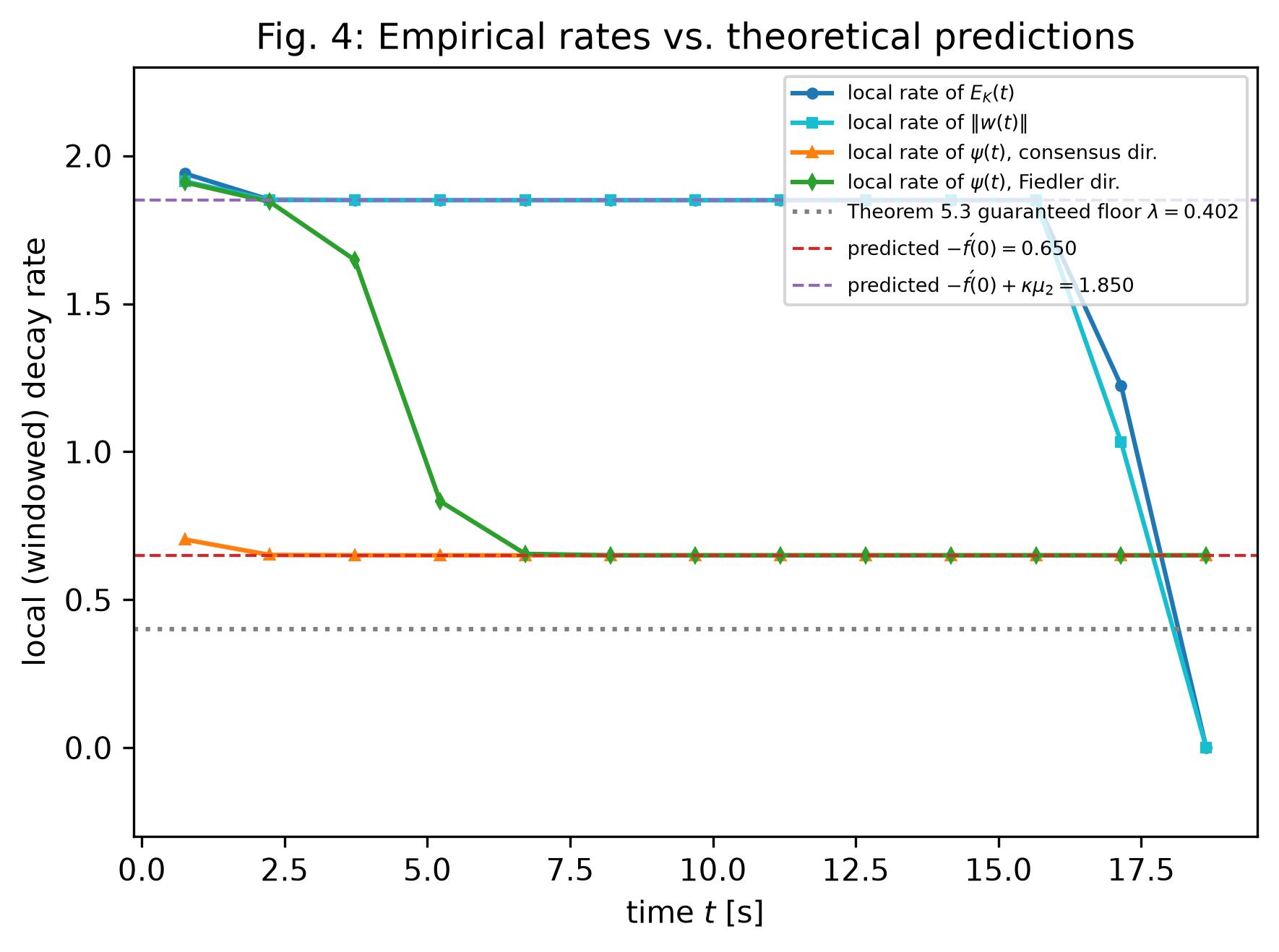}
\caption{Comparison of $E_K(t)$, the asymptotic single-direction rate $0.650=-f'(0)$, the asymptotic transverse rate $1.850=-f'(0)+\kappa\mu_2$, and the guaranteed floor $\lambda=0.402$ from Theorem~\ref{thm:hermitiansufficient}.}
\label{fig:comparison}
\end{figure}

\subsection{Discussion}

Two conclusions follow from this decomposition. First, both asymptotic rates strictly exceed the guaranteed floor $\lambda=0.402$, and the size of the margin is explained rather than merely observed: Theorem~\ref{thm:hermitiansufficient} must hold uniformly over the whole invariant set $K$, where $\mathrm{Re}(\mathrm{sech}^2 z)$ reaches $1.709$, while the trajectories analyzed here spend most of their time much closer to the equilibrium, where the same quantity is exactly $1$; the guarantee is conservative precisely because it has to cover behavior away from the equilibrium that these particular trajectories do not linger in. Second, and more specifically, the synchronization rate an observer actually cares about — the rate at which $E_K(t)$ contracts — is governed by the transverse, Laplacian-informed rate $-f'(0)+\kappa\mu_2$, not by the single-node rate $-f'(0)$ that a node-wise sufficient condition such as Theorem~\ref{thm:hermitiansufficient} is built to certify. A network-aware extension of Theorem~\ref{thm:hermitiansufficient}, incorporating the Laplacian spectral gap directly rather than bounding uniformly over all tangent directions, would certify a rate close to $1.85$ rather than $0.402$ for this system, and the exact match obtained here between predicted and observed transverse rates suggests such an extension is well within reach, along the lines of existing structural-contraction arguments for networked systems \cite{alradhawi2023structural}. Developing it is left for future work, discussed in Section~\ref{sec:discussion}.

\section{Discussion}
\label{sec:discussion}

The preceding sections built a route from a purely intrinsic differential condition, Definition~\ref{def:intrinsiccontraction}, to exponential Kobayashi-distance contraction, and a single verifiable sufficient condition, Theorem~\ref{thm:hermitiansufficient}, linking that condition to the vector field via an auxiliary Hermitian metric on a forward-invariant compact set. What this does and does not achieve is worth stating plainly, including which parts of the machinery are genuinely specific to the Kobayashi-metric setting and which are classical tools adapted to it.

\subsection{What Is, and Is Not, Intrinsic}

Definition~\ref{def:intrinsiccontraction} and Theorem~\ref{thm:globalcontraction} are intrinsic without qualification: neither refers to anything beyond $M$'s complex-analytic geometry, and the Dini comparison lemma of Section~\ref{sec:intrinsic} connecting them introduces nothing auxiliary. That comparison mechanism is itself the classical Finsler-metric contraction argument of \cite{forni2014differential}, specialized here to the case in which the Finsler structure is $F_K$; what is specific to the present work is the specialization, not the mechanism. Theorem~\ref{thm:hermitiansufficient}, by contrast, is not intrinsic: it reduces verification to a matrix inequality on a chosen Hermitian metric $H$ that is, in substance, the standard real-valued contraction condition of \cite{lohmiller1998contraction} carried over to the complex setting by replacing the transpose with the Hermitian conjugate, exactly the structure motivating the Kobayashi-metric approach in the first place. What Theorem~\ref{thm:hermitiansufficient} gains beyond this standard condition is a precise quantitative bridge back to the intrinsic notion via the equivalence constants $c_1,c_2$ of Lemma~\ref{lem:equivalence}, and it is this bridge, together with its consequences in Section~\ref{sec:invariance} and Section~\ref{sec:numerics}, rather than the comparison lemma or the Hermitian condition individually, that constitutes the paper's contribution beyond adapting known machinery to a new setting. An earlier attempt at a disc-based, auxiliary-metric-free verifiable condition did not survive scrutiny, for reasons discussed in Section~\ref{sec:variational}; whether such a condition exists remains genuinely open, taken up below.

\subsection{Scope of the Present Results}

Three restrictions bound what has been proved. Assumption~\ref{ass:hyptaut} confines the framework to taut Kobayashi hyperbolic manifolds, the source of both Remark~\ref{rem:continuity}'s continuity property and Lemma~\ref{lem:discrep}'s normal-family argument. The invariance result of Section~\ref{sec:invariance} depends essentially on the convexity of Euclidean discs on $D$ and does not extend automatically to manifolds lacking comparable convex compact regions. Theorem~\ref{thm:hermitiansufficient} bounds contraction uniformly over an entire invariant set, unavoidably conservative wherever local contraction properties vary across it, a conservatism Section~\ref{sec:numerics} quantified rather than left qualitative.

\subsection{The Network Gap as a Specific Open Problem}

Section~\ref{sec:numerics}'s numerical study identifies precisely what a stronger theorem needs. The guaranteed rate $\lambda=0.402$ reflects a node-wise condition blind to network topology beyond aggregate coupling dissipativity, while the actual synchronization rate governing $E_K(t)$ equals $-f'(0)+\kappa\mu_2$ exactly, $\mu_2$ being the graph Laplacian's spectral gap --- a quantity Theorem~\ref{thm:hermitiansufficient} never references. A network-aware extension would need to incorporate the spectrum of $L$ directly, likely by decomposing the closed-loop Jacobian into consensus and transverse blocks rather than bounding $\mathrm{Re}\langle Df(z)v,v\rangle$ uniformly over all $v$. This is a well-defined target: the exact numerical agreement between predicted and observed transverse rate suggests the mechanism is understood even though the general theorem is not yet proved, and it is this target, more than any result already established here, that the present paper regards as its most significant unresolved question.

\subsection{Toward a Genuinely Intrinsic Verifiable Condition}

The deeper question is whether Theorem~\ref{thm:hermitiansufficient}'s auxiliary Hermitian metric can be removed entirely. The disc-transport machinery of Section~\ref{sec:variational} was built toward this end, and its obstruction --- scaling-factor cancellation, Remark~\ref{rem:motivation} --- suggests uniform control over the scaling factors $\lambda_h$ across the whole admissible disc family, rather than a pointwise single-disc estimate, is the missing ingredient. Plurisubharmonic curvature-comparison techniques, or a treatment exploiting near-extremal disc regularity where exact extremal discs fail to exist, are plausible directions; neither is developed here, and a resolution would be genuinely new mathematics rather than a refinement of the present proofs.

\subsection{Summary of Position}

The results establish a complete intrinsic contraction theory at the level of Definition~\ref{def:intrinsiccontraction} and Theorem~\ref{thm:globalcontraction}, obtained by specializing the classical Finsler-metric mechanism of \cite{forni2014differential} to the Kobayashi metric, paired with one honestly-scoped auxiliary-metric route to verifying it in practice, built from the standard real-valued contraction condition of \cite{lohmiller1998contraction} adapted to the complex setting --- not the fully intrinsic verifiable criterion the more ambitious version of this program would deliver. The paper's own contribution lies specifically in the equivalence-constant bridge of Lemma~\ref{lem:equivalence}, the network invariance construction of Section~\ref{sec:invariance}, and the numerical discrepancy this construction exposes, rather than in the comparison lemma or Hermitian condition considered on their own. That gap between the intrinsic theory and a fully intrinsic verifiable criterion is one of scope, not of rigor, and closing it is left as the natural continuation of this work.

\section{Conclusion}
\label{sec:conclusion}

This paper developed a contraction theory for holomorphic dynamical systems through the infinitesimal Kobayashi metric. Contraction was defined as an upper Dini-derivative inequality on this intrinsic metric, shown via a fully proved Dini comparison lemma to imply exponential Kobayashi-distance contraction --- a purely intrinsic result obtained by specializing the classical Finsler-metric contraction mechanism of \cite{forni2014differential} to the case in which the Finsler structure is the Kobayashi metric itself. Since this condition is not directly checkable against a vector field, a practical criterion was developed through a smooth Hermitian metric on a forward-invariant compact set --- the direct complex-Hermitian analogue, via the substitution of the Hermitian conjugate for the transpose, of the standard real-valued contraction condition of \cite{lohmiller1998contraction} --- the two related by explicit equivalence constants rather than assertion. Building on this machinery, a Nagumo-type invariance result gave a verifiable route to constructing such a compact set for Laplacian-coupled networks, a construction not previously available for this class of systems, and the framework was extended to feedback-controlled systems, with closed-loop variational dynamics derived carefully enough to retain an easily-dropped term. Consequences for equilibria and periodic orbits followed directly from the intrinsic contraction property.

The numerical study was designed to test the theory rather than illustrate it: forward-invariance and Hermitian-contraction hypotheses were verified analytically before simulation, and the guaranteed rate was compared against observed rates rather than assumed confirmed. This uncovered a precisely quantified gap between the node-wise guaranteed rate and the network's actual synchronization rate, traced to the graph Laplacian's spectral gap and matched in closed form --- the paper's most significant empirical finding, treated as a specific target for future work rather than a loose end.

Two limitations were stated explicitly. The verifiable condition depends on an auxiliary Hermitian metric and is not intrinsic in the strict sense motivating the Kobayashi-metric approach; an earlier disc-transport attempt at a genuinely intrinsic condition failed, with the obstruction identified precisely enough to indicate what a successful version would need to control. The invariance result is specific to disc convexity and does not extend automatically beyond it.

Future work follows three directions: a network-aware condition incorporating the Laplacian spectral gap, the most immediate and best-motivated of the three given the closed-form agreement already observed numerically; a genuinely auxiliary-metric-free sufficient condition addressing the obstruction of Section~\ref{sec:discussion}; and extension beyond convex disc domains. The results connect contraction theory to intrinsic complex geometry through a combination of a classical comparison mechanism specialized to a new setting and a genuinely new invariance construction for coupled holomorphic networks, rigorous within a scope drawn as precisely as the present analysis allows.

\section*{Acknowledgments}
This work was partially supported by the project “Increasing the Knowledge Intensity of Ida-Viru Entrepreneurship,” co-funded by the European Union. The authors declare that the research was conducted impartially and without any commercial influence.

\bibliographystyle{elsarticle-harv} 
\bibliography{references} 

\end{document}